\documentclass[reqno,12pt,centertags, draft]{amsart}
\usepackage[left=1.2in,right=1.2in,bottom=1.2in,top=1.2in]{geometry}
\usepackage{hyperref}
\usepackage{amsmath,amsthm,amscd,amssymb,amsfonts,latexsym,upref,stmaryrd,
	enumerate,color,verbatim,mathrsfs, mathtools}
\usepackage{tikz}
\usepackage{cite}
\usetikzlibrary{trees}
\date{\today}

\mathtoolsset{showonlyrefs}

\newcommand*{\mailto}[1]{\href{mailto:#1}{\nolinkurl{#1}}}

\newcommand{\R}{{\bbR}}

\newcommand{\Z}{{\bbZ}}
\newcommand{\C}{{\bbC}}

\newcommand{\bbC}{{\mathbb{C}}}

\newcommand{\bbN}{{\mathbb{N}}}

\newcommand{\bbR}{{\mathbb{R}}}

\newcommand{\bbZ}{{\mathbb{Z}}}

\newcommand{\cA}{{\mathcal A}}
\newcommand{\cB}{{\mathcal B}}

\newcommand{\cD}{{\mathcal D}}

\newcommand{\cK}{{\mathcal K}}

\newcommand{\cM}{{\mathcal M}}

\newcommand{\bfi}{{\bf i}}

\DeclareMathOperator{\sign}{sign}

\DeclareMathOperator{\supp}{supp}

\DeclareMathOperator{\ran}{ran}
\DeclareMathOperator{\dom}{dom}

\DeclareMathOperator{\tr}{tr}

\DeclareMathOperator*{\conn}{conn}
\DeclareMathOperator*{\sep}{sep}
\DeclareMathOperator*{\triv}{triv}

\newcommand{\SL}{\mathrm{SL}}

\newcommand{\no}{\notag}
\newcommand{\lb}{\label}

\newcommand{\wti}{\widetilde}

\newcommand{\hatt}{\widehat}
\newcommand{\bi}{\bibitem}

\renewcommand{\ge}{\geqslant}
\renewcommand{\le}{\leqslant}

\let\geq\geqslant
\let\leq\leqslant

\definecolor{purple}{rgb}{.5,0,1}

\newcommand{\set}[1]{{\left\{ {#1} \right\}}}

\makeatletter
\def\theequation{\@arabic\c@equation}

\allowdisplaybreaks
\numberwithin{equation}{section}

\newtheorem{theorem}{Theorem}[section]

\newtheorem{lemma}[theorem]{Lemma}

\newtheorem{hypothesis}[theorem]{Hypothesis}

\newtheorem{step}{Step}
\theoremstyle{remark}
\newtheorem{remark}[theorem]{Remark}

\theoremstyle{definition}
\newtheorem{case}{Case}
\newtheorem{subcase}{Case}
\numberwithin{subcase}{case}

\begin{document}
	
	\numberwithin{equation}{section}
	\allowdisplaybreaks
	
	\title[Localization for Singularly Perturbed Laplacians]{Random Hamiltonians with Arbitrary Point Interactions}

	\author[D.\ Damanik]{David Damanik}
	\address{ Department of Mathematics, Rice University, Houston, TX 77005, USA}
	\email{\mailto{damanik@rice.edu}}
	\thanks{D.\ D.\ was supported in part by NSF grant DMS--1700131 and by an Alexander von Humboldt Foundation research award.}
		
	\author[J. Fillman]{Jake Fillman}
	\address{Department of Mathematics, Texas State University, San Marcos, TX 78666, USA}
	\email{\mailto{fillman@txstate.edu}}
		
	\author[M.\ Helman]{Mark Helman}
		\address{ Department of Mathematics, Rice University, Houston, TX 77005, USA}
	\email{\mailto{mh84@rice.edu}}
	\author[J.\ Kesten]{Jacob Kesten}
		\address{ Department of Mathematics, Rice University, Houston, TX 77005, USA}
	\email{\mailto{jgk3@rice.edu}}
	\author[S.\ Sukhtaiev]{Selim Sukhtaiev}
	\address{ Department of Mathematics, Rice University, Houston, TX 77005, USA}
	\email{\mailto{sukhtaiev@rice.edu}}
	\thanks {S.S.\ was supported in part by an AMS-Simons travel grant, 2017-2019}
	

	\date{\today}


	\vspace*{-3mm}
	
	\maketitle
	
	\begin{abstract}
	 We consider disordered Hamiltonians given by the Laplace operator subject to arbitrary random self-adjoint singular perturbations supported on random discrete subsets of the real line. Under minimal assumptions on the type of disorder, we prove the following dichotomy: Either every realization of the random operator has purely absolutely continuous spectrum or spectral and exponential dynamical localization hold.  In particular, we establish Anderson localization for Schr\"odinger operators with Bernoulli-type random singular potential and singular density.
	\end{abstract}

\tableofcontents

	\section{Introduction}

	\subsection{Overview} The spectral theory of Schr\"odinger operators with singular potentials, originally motivated by the Kronig--Penney model from solid state physics, has been of interest since at least 1961 when Berezin and Faddeev \cite{BF} gave a mathematically rigorous treatment of $-\Delta+\varepsilon\delta$, where $\varepsilon$ is a real parameter and $\delta$ denotes a Dirac delta distribution. An illuminating discussion of this subject together with historical remarks and relevant references can be found in the classical monograph \cite{AGHH}.
	
	The main focus of this paper is on Anderson localization for random Hamiltonians with \emph{arbitrary} point interactions under \emph{minimal assumptions} on the randomness. The first relevant work in this direction is due to Delyon, Simon, and Souillard \cite{DSS}. They established spectral localization for
	$-\Delta+\sum_{j\in\bbZ} \lambda_j(\omega)\delta(x-j)$,
	where $\{\lambda_j\}_{j\in\bbZ}$ is a sequence of independent identically distributed random variables whose common distribution has a sufficiently regular nontrivial absolutely continuous part. More recently, Hislop, Kirsch, and Krishna \cite{HKK05}, \cite{HKK19} proved Anderson localization (in suitable energy regions) and studied eigenvalue statistics for the same model in dimensions $d=1,2,3$. Localization and zero-measure spectrum for closely related quantum graph models were established in \cite{DLS, DFS}.
	
	The principal achievement of this paper is twofold. First, we cover arbitrary (as discussed in \cite[Section 3.4]{K96}) self-adjoint second order differential operators with coefficients supported on a discrete set $\{t_j\}_{j\in\bbZ}$.  Similar to the Kronig--Penney model, these operators are realized via self-adjoint vertex conditions imposed at every $t_j$. Second, we make no assumptions on the regularity of the common probability distribution of i.i.d.\ random variables in question, contrary to all previously considered Kronig--Penney type random models. Such a level of generality is essential in several random quantum graph models where the random variables take integer values representing geometric characteristics of graphs, e.g., the number of edges, cf. \cite{DFS}.
	
	The main ingredient of the proof is the fact that the Lyapunov exponent is positive away from a discrete set of exceptional energies, which we establish in Theorem \ref{t:posExps}. It is worth noting that the underlying one-step transfer matrix takes a rather general form given by a product of the monodromy matrix of the free Hamiltonian and an {\it arbitrary} $\SL(2,\bbR)$ matrix, see \eqref{2.2}. The latter describes the general self-adjoint vertex condition mentioned above, which takes the form
	\begin{align}\lb{new1.1}
	 \begin{bmatrix}
	u(t_j^+)\\
	u'(t_j^+)
	\end{bmatrix}
	=
	B_j \begin{bmatrix}
	u(t_j^-)\\
	u'(t_j^-)
	\end{bmatrix}, \quad B_j\in\SL(2,\bbR).
	\end{align}
	Having established positivity of Lyapunov exponents, we proceed with the proof of localization following \cite{BuDaFi} and its continuum versions \cite{BuDaFi2, DFS}; see Theorem~\ref{main1}.

	\subsection{Main result}
	To begin, we discuss self-adjoint realizations of the Laplace operator subject to singular perturbations supported on a uniformly discrete set of vertices
	\begin{equation}\lb{1.1b}
	\{t_j\}_{j\in\Z}\subset \R,\quad \inf\limits_{j\in\bbZ}(t_{j+1}-t_j)>0.
	\end{equation}
	Let $H_{\min}$ be the operator acting in $L^2(\R)$ and given by
	\begin{align}
	&H_{\min}u:=-u'', \ u\in\dom(H_{\min}),\\
	&\dom(H_{\min}):= \{u\in \hatt H^2(\R): u(t_j^{\pm})=u'(t_j^{\pm})=0, \ j\in\bbZ\},
	\end{align}
	where $\hatt H^2(\R):=\oplus_{j\in\Z} H^2(t_j, t_{j+1})$ denotes the direct sum of Sobolev spaces.
	This operator is symmetric and has infinite deficiency indices. Its adjoint $H_{\max}:=H_{\min}^*$ is given by
\[
H_{\max}u:=-u'', \ u\in\dom(H_{\max})=\hatt H^2(\R),
\]	
see \cite[Section III.2.1]{AGHH}. All self-adjoint extensions $H$ of $H_{\min}$ (automatically satisfying $H_{\min}\subset H=H^*\subset H_{\max}$) can be described by means of vertex conditions imposed at every $t_j$.  Define
	\begin{align}
	&{\mathscr M}_{\triv}:=\left\{\left(\begin{bmatrix}\lb{1.3}
	1&0\\
	0&1
	\end{bmatrix},\, e^{\bfi\theta}\begin{bmatrix}
	1&0\\
	0&1
	\end{bmatrix} \right)\Big| \theta\in[0,2\pi) \right\}, \\
	&{\mathscr M}_{\conn}:=\left\{\left(\begin{bmatrix}
	1&0\\
	0&1
	\end{bmatrix}, e^{\bfi\theta}\begin{bmatrix}
	\alpha&\beta\\
	\gamma&\delta
	\end{bmatrix} \right)\Big|\hspace{.3cm} \begin{matrix}
\theta\in[0,2\pi), \\
\{\alpha, \beta, \gamma, \delta\}\subset\R, \\ \alpha\delta-\gamma\beta=1
	\end{matrix}\right\},\lb{1.1d}\\
	&{\mathscr M}_{\sep}:=\left\{\left(\begin{bmatrix}
	x&y\\
	0&0
	\end{bmatrix}, \begin{bmatrix}
	0&0\\
	w&z
	\end{bmatrix} \right)\Big|\  \{x,y,w,z\}\subset\R \right\},\\
	&\mathscr M:={\mathscr M}_{\conn}\cup{\mathscr M}_{\sep}. \lb{1.3a}
	\end{align}
	Then by \cite[Theorem 1]{ADK} (see also \cite{AFK, CH, K96, Seba}) the operator  $H=(H_{\max})\upharpoonright_{\dom(H)}$ is a self-adjoint extension of $H_{\min}$ if and only if there exists a sequence
	\begin{equation}\lb{1.1a}
	\{(A_j, B_j)\}_{j\in\Z}\subset\mathscr M,
	\end{equation}	
	such that $\dom(H)=\{ u\in \hatt H^2(\bbR): u\text{\ satisfies \eqref{1.1}} \}$
	\begin{align}\lb{1.1}
	A_j \begin{bmatrix}
	u(t_j^+)\\
	u'(t_j^+)
	\end{bmatrix}
	=
	B_j \begin{bmatrix}
	u(t_j^-)\\
	u'(t_j^-)
	\end{bmatrix}, j\in\bbZ.
	\end{align}
	
	The main goal of this paper is to show that a random choice of vertex conditions \eqref{1.1a}, \eqref{1.1} leads to Anderson localization unless the vertex matrices are drawn in such a way that the resulting Hamiltonians are all unitarily equivalent to an operator with periodic coefficients,	in which case the resultant operators all have purely absolutely continuous spectrum.
	
	  The location of vertices supplies another source of randomness in our model. We assume that the distance $\ell_j$ between $t_{j-1}$ and $t_{j}$ is random. To facilitate this, for a sequence $\{\ell_j\}_{j\in\bbZ}\subset (0,\infty)$, we denote $t_0:=0$ and
	\begin{equation}\lb{1.2}
	t_j:=\begin{cases}
	\sum_{k=1}^{j}\ell_k& j\geq 1,\\
	-\sum_{k=j+1}^{k=0}\ell_k,& j\leq -1.
	\end{cases}
	\end{equation}
	
	\begin{hypothesis}\lb{hyp1.2}
	Fix $L^+\geq L^->0$. Suppose that $\mathscr A\subset [L^-,L^+]\times\mathscr M$ is a bounded set. Let $\wti \mu$ be an arbitrary probability measure on $\mathscr A$ and let $(\Omega, \mu):=(\mathscr A, \wti\mu)^{\bbZ}$.
	\end{hypothesis}

For a sequence
\[
\omega=\{\ell_j,  (A_j, B_j)\}_{j\in\bbZ}\in\Omega,
\]
let $H_{\omega}$ denote the self-adjoint extension of $H_{\min}$ corresponding to the discrete set of vertices $\{t_j\}_{j\in\bbZ}$ given by \eqref{1.2} and the boundary conditions \eqref{1.1}.

	\begin{theorem}\lb{thm2.3}Assume Hypothesis \ref{hyp1.2}.  	
			{\rm(i)} Suppose that there exist \[(\ell_1,  (A_1, B_1)), (\ell_2,  (A_2, B_2))\in \supp \wti \mu\] such  one of the following holds
	\begin{itemize}
	\item $\ell_1\not=\ell_2, A_1=A_2=I_2, B_1\not\in \{e^{\bfi \theta}I_2: \theta\in[0, 2\pi) \},$
	\item $A_1=A_2=I_2, B_1\not\in \{e^{\bfi \theta}B_2:  \theta\in[0, 2\pi)\},$
	\item  $(A_1, B_1)\in \mathscr M_{\sep}$.
	\end{itemize}
	Then $H_{\omega}$ possesses a basis of exponentially decaying eigenfunctions for $\mu$-almost every $\omega\in\Omega$. Furthermore,  there exist a set $ \Omega^*\subset\Omega$ with $\mu(\Omega^*)=1$ and  a discrete set $\mathfrak{D} \subseteq \R$ such that for every compact interval $I\in\bbR\setminus \mathfrak{D}$, every $p>0$ and every compact set $\cK\subset \bbR$,
	\begin{equation}
	\sup\limits_{t>0} \left\||X|^p\chi_I(H_{\omega}) e^{-itH_{\omega}}\chi_{\cK}\right\|_{L^2(\bbR)}<\infty,\ \omega\in\Omega^*,
	\end{equation}
	where $\chi_I(H_{\omega})$ denotes the spectral projection corresponding to $I$, and $|X|^p $ denotes the operator of multiplication by the function $f(x):=|x|^p$.
	
	{\rm(ii)} If the assumptions of part {\rm(i)} are not satisfied then $H_{\omega}$ has purely absolutely continuous spectrum for every $\omega \in \Omega$.
	\end{theorem}

\begin{remark}\lb{newrem}\begin{enumerate}
\item Part~(ii) of Theorem~\ref{thm2.3} follows from general arguments. For example, if every  $(\ell,(A,B)) \in \supp\wti\mu$ satisfies $A = I_2$, $B = e^{\bfi\theta}I_2$ for some $\theta \in [0,2\pi)$, then all realizations of $H_\omega$ will be unitarily equivalent to the free Laplacian, and hence will exhibit purely absolutely continuous spectrum. Similarly, if there exist $\ell > 0$ and $B \in \SL(2,\R)$ such that all elements of $\supp\wti\mu$ are of the form $(\ell,(I_2,e^{\bfi\theta}B))$ for some $\theta$, then every realization of $H_\omega$ will be unitarily equivalent to an operator with periodic point interactions and again the desired localization fails.

In particular, we want to point out that Part~(i) is optimal in the sense that any amount of randomness that pushes one outside the periodic case will produce spectral and dynamical localization.
\medskip

\item In the third case of Theorem~\ref{thm2.3}.(i) (that is, when $\supp\wti\mu \cap \mathscr{M}_{\sep} \neq \emptyset$), $H_\omega$ decouples into an infinite direct sum of operators on finite intervals ($\mu$-almost surely). These operators have compact resolvents by general arguments, so the associated spectra are pure point with compactly supported (hence exponentially decaying) eigenfunctions. Moreover, in this case one has $\mathfrak D=\emptyset$.

\end{enumerate}

\end{remark}

	Let us point out that this result yields localization for several physically relevant Hamiltonians. Let $\{\alpha_j\}_{j\in\bbZ}$ be a sequence of independent identically distributed random variables taking at least two distinct values in a bounded subset of $\bbR$. Define
	$
	A(x):=\sum_{j\in\bbZ} \alpha_j\delta(x-j),
	$
	and introduce formally self-adjoint differential expressions
	\begin{align}
	&\tau_S:=-D_x^2+A,\ \tau_D:=-D_x(1+A)D_x,\ \tau_G:=(\bfi D_x+A)^2-A^2.
	\end{align}
	As was shown in \cite{K96} these differential expressions may be realized as self-adjoint  extensions of $H_{\min}$ corresponding to $\ell_j \equiv 1$ (i.e.\ $t_j=j$) and the following vertex matrices
	\begin{align}
	&\tau_S\sim (A_j, B_j)= \left(\begin{bmatrix}
	1&0\\
	0&1
	\end{bmatrix},\begin{bmatrix}
	1&0\\
	\alpha_j&1
	\end{bmatrix} \right)\\
	&\tau_D\sim (A_j, B_j)= \left(\begin{bmatrix}
	1&0\\
	0&1
	\end{bmatrix},\begin{bmatrix}
	1&-\alpha_j\\
	0&	1
	\end{bmatrix} \right)\\
		&\tau_G\sim (A_j, B_j)= \left(\begin{bmatrix}
	1&0\\
	0&1
	\end{bmatrix},\begin{bmatrix}
	\frac{2+\bfi \alpha_j}{2-\bfi \alpha_j}&0\\
	0&	\frac{2+\bfi \alpha_j}{2-\bfi \alpha_j}
	\end{bmatrix} \right)
	\end{align}
	The first and the second cases satisfy the assumptions of Theorem \ref{thm2.3} part (i), the third case satisfies the assumptions of Theorem \ref{thm2.3} part (ii). Hence, the first two operators exhibit Anderson localization, while the third operator has purely absolutely continuous spectrum.
	
	Another relevant application is to Anderson localization for the Kirchhoff Laplacian on random radial trees discussed in \cite{DFS}, \cite{HiPo}. In this case the Hamiltonian is determined by the following matrices,
	\begin{equation}
		(A_j, B_j)= \left(\begin{bmatrix}
		1&0\\
		0&1
		\end{bmatrix},\begin{bmatrix}
		\sqrt{\beta_j}&0\\
		\frac{\alpha_j}{\sqrt{\beta_j}}&\frac{1}{\sqrt{\beta_j}}
		\end{bmatrix} \right),
	\end{equation}
	where $\{\alpha_j\}\subset \bbR$, $\{\beta_j\}\subset\bbN$ are sequences of i.i.d. random variables taking at least two distinct values in a bounded subset of $\bbR$. Anderson localization for this model was proved in \cite{DFS} and also follows from Theorem \ref{thm2.3}.

	\section{Ergodic Setup and Positive Lyapunov Exponents}

  In this section we assume Hypothesis \ref{hyp1.2} with the additional restriction
	\begin{equation}\lb{2.1}
\mathscr A\subset  [L^-,L^+]\times \{I_2\}\times \SL(2,\bbR).
	\end{equation}
We will explain in the beginning of Section~\ref{Section3} that it is enough to prove Theorem~\ref{thm2.3} assuming \eqref{2.1}.

 \subsection{Ergodic Setup}

Let us discuss the eigenvalue problem $Hu=Eu$, where $H$ is a self-adjoint extension of $H_{\min}$  corresponding to a fixed set of vertices \eqref{1.1b} and vertex conditions \eqref{1.1} with $\{(A_j, B_j)\}_{j\in\bbZ}$ satisfying
\[A_j= I_2,\ B_j\in{\rm SL}(2,\bbR),\  j\in\bbZ.
\]
Consider the differential equation $-f''=Ef$ subject to $f\in H^2(t_j, t_{j+1})$, $j\in\bbZ$ and vertex conditions \eqref{new1.1}. The solution $f$ satisfies
	\begin{align} \label{eq:halfLineTMdiscrete}
	&\begin{bmatrix}
	f(t_j^+)\\
	f'(t_j^+)
	\end{bmatrix}=\cM^{E}(\ell_j, B_j)
	\begin{bmatrix}
	f(t_{j-1}^+)\\
	f'(t_{j-1}^+)
	\end{bmatrix},\ j\in\bbZ,
	\end{align}
	where the mapping $\cM^E :  \mathscr A \to \SL(2,\R)$\footnote{By \eqref{2.1} the second component of all elements of $\cA$ is $I_2$. Consequently, slightly abusing notation, we may view $\cM^E$ as a function of  two variables  $(\ell,B)$.} is defined by
\begin{equation}\lb{2.2}
	\cM^{E}(\ell, B):=B \begin{bmatrix}
\cos\sqrt E\ell &\frac{\sin\sqrt E \ell}{\sqrt E} \\
-\sqrt E \sin\sqrt E\ell & \cos\sqrt E\ell
\end{bmatrix}.
\end{equation}
We note that the entries of $\mathcal  M^E(\ell,B)$ are well-defined analytic functions of $E \in \C$.
	
Define  $M^E: \Omega \to \SL(2,\R)$ by  $M^E(\omega):= \cM^E(\omega_1)$.
Let $T$ denote the left shift acting on $\Omega$ and define the skew product
	\begin{equation}\no
	(T,M^E): \Omega\times\bbR^2\rightarrow  \Omega\times\bbR^2,\ (T,M^{ E})(\omega, v)=(T\omega, M^E(\omega)v).
	\end{equation}
We denote the $n$-step transfer matrix by
\begin{equation}\lb{tm}
	M^{E}_n(\omega)=\prod_{r=n-1}^0M^{E}(T^r\omega)
	=
	M^E(T^{n-1}\omega)\cdots  M^E(T\omega)  M^E(\omega),\ n\in\bbN,\
\end{equation}
and note that the iterates over the skew product are given by $(T,M^{E})^n=(T^n, M^{E}_n)$. One has
\[
\begin{bmatrix}
u(t_n^+)\\ u'(t_n^+)
\end{bmatrix}
=
M^E_n(\omega)
\begin{bmatrix}
u(0^+) \\ u'(0^+)
\end{bmatrix}
\text{ for all }n \in \Z
\]
whenever $u \in \hatt H^2(\R)$ satisfies $-u''=Eu$ and the vertex conditions from \eqref{new1.1} corresponding to $H_\omega$. The Lyapunov exponent is defined by
	\begin{equation}\lb{LE1}
	L(E):=\lim\limits_{n\rightarrow\infty}\frac{1}{n}\int_{\Omega}\log \|M_n^{E}(\omega)\| \, d\mu(\omega).
	\end{equation}
	By Kingman's Subadditive Ergodic Theorem we have
	\begin{equation}\lb{LE2}
	L(E)=\lim\limits_{n\rightarrow\infty}F_n(\omega,E) ,
	\end{equation}
	for $\mu$-almost every $\omega$, where $F_n(\omega,E):=\frac1n\log\|M^E_n(\omega)\|$. Let us point out that there is also a natural continuum cocycle which satisfies
	\[
	\overline{M}^E_x(\omega)\begin{bmatrix}
u(0^+) \\ u'(0^+)
\end{bmatrix}=\begin{bmatrix}
u(x^+) \\ u'(x^+)
\end{bmatrix}
	\]
	in the event that $u$ solves $-u''=Eu$ and satisfies the vertex conditions corresponding to $H_\omega$. By a simple application of Birkhoff's ergodic theorem, the Lyapunov exponent for this cocycle is related to that of the discrete cocycle via
	\[
L(E) = \overline{\ell}\cdot \overline{L}(E),
	\]
	where $\overline\ell: = \int_{\mathscr{A}} \alpha_1 \, d\wti\mu(\alpha)$ denotes the $\wti\mu$-expected value of the length.

	\subsection{Positivity of Lyapunov Exponents}

\begin{hypothesis}\lb{hyp2.1}
Assume Hypothesis \ref{hyp1.2} with
\begin{equation}
\mathscr A\subset  [L^-,L^+]\times \{I_2\}\times \SL(2,\bbR).
\end{equation}
 Suppose that there exist $B_1,B_2\in\SL(2,\bbR)$, $\ell_1, \ell_2\in[L^-, L^+]$ such that
\begin{equation}\lb{2.3b}
(\ell_1,B_1) \not= (\ell_2,B_2), \text{ and }
B_j \neq I_2 \text{ for some } j=1,2,
\end{equation}
\begin{equation}
(\ell_j, I_2, B_j)\in\supp \wti \mu, \quad j=1,2.
\end{equation}
\end{hypothesis}
The main assertion of this subsection is that the Lyapunov exponent is positive away from a discrete set of exceptional energies.

\begin{theorem}\label{t:posExps}
	Assume Hypothesis \ref{hyp2.1}. Then there is a discrete set $\mathfrak D\subseteq \R$ with the property that $L(E) > 0$ for every $E \in \R\setminus \mathfrak D$.
\end{theorem}

	To begin, we address the key technical fact that will be utilized in the proof of Theorem~\ref{t:posExps}. Given $(\ell_j,I_2,B_j)\in \mathscr A$, $j=1,2$, and $E \in \C$, define
	\[
	G(E)
	= G((\ell_1,B_1),(\ell_2,B_2),E)
	= \left[\mathcal M^E(\ell_1,B_1) , \mathcal M^E(\ell_2,B_2)\right]
	\]
	to be the commutator of the two matrices $\mathcal  M^E(\ell_j,B_j)$, $j=1,2$. In view of \cite{BuDaFi2}, the key obstruction to positive exponents away from a discrete set of energies is everywhere vanishing of $G$.

	\begin{theorem} \label{t:main}
	
	Given $(\ell_j,I_2,B_j) \in \mathscr A$, one has $G((\ell_1,B_1),(\ell_2,B_2),E) = 0$ for all $E \in \C$ if and only if at least one of the following statements is true:
		\begin{align}
	  &\ell_1  =\ell_2 \text{ and } B_1 \in\{B_2, -B_2\}\\
	&\{B_1, B_2\} \subset \{I_2, -I_2\}.
		\end{align}
	\end{theorem}
	The following lemma will be helpful in the proof of Theorem~\ref{t:main}.
	
	\begin{lemma} \label{lem:triglim}
		If $P,Q,R,S\in\R$ with $R>0,S>0$ and
		\begin{equation} \label{eq:lem:triglim}
		\lim_{w\to\infty} P \cos(S w) \sin(R w) + Q\cos(R w) \sin(S w)=0,
		\end{equation}
		then either $P=Q=0$ or $R=S$ {\rm(}and $P+Q=0${\rm)}. In particular, if \eqref{eq:lem:triglim} holds true, then
		\[
		P\cos(Sw)\sin(Rw) + Q\cos(Rw)\sin(Sw) = 0 \text{ for all } w.
		\]
	\end{lemma}
	\begin{proof}
		Since $P \cos(S w) \sin(R w) + Q\cos(R w) \sin(S w) =:f(w)$ is an almost-periodic function of $w$, if \eqref{eq:lem:triglim} holds, then $f$ vanishes identically. To see this, note that if $|f(y)| = \delta>0$ for some $y \in \R$ and $p$ is a $\delta/2$-almost period of $f$, then $|f(y+kp)| \ge \delta/2$ for all $k \in \Z$.
		
		Thus, we assume $f \equiv 0$ and that $P,Q \neq 0$. Since $f(\pi/S) = 0$, we arrive at
		\[
		P \sin(\pi R/S) = 0.
		\]
		Since $P \neq 0$, this yields $\sin(\pi R/S) = 0$, hence $R/S \in \Z$. Interchanging the roles of $R$ and $S$ implies $S/R \in \Z$ as well. Since $S,R>0$, this forces $R=S$ and hence $P+Q=0$ as well.
	\end{proof}
	
	With Lemma~\ref{lem:triglim} in hand, we prove Theorem~\ref{t:main}.
	
	\begin{proof}[Proof of Theorem~\ref{t:main}]

	Given $t\in[0,\pi)$, $b>0$, $\ell>0$, $q\in\R$, and $E\in\C$, introduce $w = \sqrt{E}$, and define the matrices
	\begin{equation}\lb{2.3}
		R(t) = \begin{bmatrix}\cos(t)&-\sin(t)\\\sin(t)&\cos(t)\end{bmatrix},
	\quad
	D(b) = \begin{bmatrix}b&0\\0&1/b\end{bmatrix},
	\quad
	S(q)
	=
	\begin{bmatrix}1&0\\q&1\end{bmatrix},
	\end{equation}
	and
	\begin{equation}\lb{2.4}
	T(w,\ell)=\begin{bmatrix}\cos(\ell w)&\sin(\ell w)/w\\-w\sin(\ell w)&\cos(\ell w)\end{bmatrix}.
	\end{equation}
	Consider the matrix given by
	\begin{equation}\lb{2.5}
\widetilde M(t,b,\ell,q,w)
=
R(t)D(b)S(q)T(w,\ell).
	\end{equation}
For brevity, we introduce
	\[
	\mathcal{A} = [0,\pi) \times (0,\infty) \times (0,\infty) \times \R
	\]
	for the parameter space; given $\alpha = (t,b,\ell,q) \in \mathcal A$, we abuse notation a bit and write $\widetilde M(\alpha,w)$ for $\widetilde M(t,b,\ell,q,w)$. Since $B_j \in \SL(2,\R)$, the Iwasawa decomposition (see, e.g.\ \cite{LangSL2R}) implies there exist $t_j$, $b_j$, and $q_j$ such that $\alpha_j := (t_j,b_j,\ell_j,q_j) \in \mathcal A$ and $B_j =\pm R(t_j) D(b_j) S(q_j)$, and hence $\widetilde M(\alpha_j,w) = \pm M^E(\ell_j,B_j)$. It is enough to assume the representation $B_j =R(t_j) D(b_j) S(q_j)$ and prove that one of
	\begin{align}
	\label{eq:Gequiv0:1}      \ell_1 & =\ell_2 \text{ and } B_1 = B_2,\\
	\label{eq:Gequiv0:2}   B_1 & = B_2 = I_2
	\end{align}
holds.

		If either \eqref{eq:Gequiv0:1} or \eqref{eq:Gequiv0:2} holds, then a straightforward calculation reveals that $G(E) \equiv 0$. Conversely, assume that $G\equiv 0$; in particular, $G_{1,1}\equiv 0$.
		
		\begin{case} {\boldmath $\ell_1 = \ell_2$\textbf{.}}  In this case, we denote $\ell := \ell_1=\ell_2$. \end{case}
		
		\begin{subcase}{\boldmath $t_1,t_2\neq 0$\textbf{.}} \end{subcase}
		
		Since $\sin t_1,\sin t_2\neq 0$, after making the substitutions $x_j = \cot(t_j)$ and $z = \ell w$, we obtain
		\[
		0
		\equiv \frac{b_1b_2}{\sin(t_1)\sin(t_2)}G_{1,1}
		= A\cos^2z+B\sin^2z+Cz\cos z\sin z+D\frac{\cos z\sin z}z,\]
		where
		\begin{align*}
		A &= b_1^2-b_2^2+q_1x_1-q_2 x_2\\
		B &= q_1 x_2-q_2 x_1 +  (b_2^2 - b_1^2) x_1x_2\\
		C &= \frac{x_2-x_1}{\ell}\\
		D&= \ell\Big( \left(b_1^2b_2^2 +q_1q_2\right) (x_1-x_2) + \left(b_1^2 q_2 - b_2^2 q_1\right)(1+x_1x_2) \Big).
		\end{align*}
		Since the set of functions $\{\cos^2 z, \sin^2 z, z \sin(2z), z^{-1} \sin(2z)\}$ is linearly independent over $\C$, we have $A = B = C = D = 0$, and thus we arrive at
		\begin{align}
		0 & = b_1^2-b_2^2+q_1x_1-q_2x_2 \label{eq:Gequiv0:A=0} \\
		0 & =-q_2 x_1 + q_1 x_2 - b_1^2 x_1 x_2 + b_2^2 x_1 x_2 \label{eq:Gequiv0:B=0}\\
		0 & =\frac{-x_1+x_2}{\ell} \label{eq:Gequiv0:C=0} \\
		0 & = \ell\Big( \left(b_1^2b_2^2 +q_1q_2\right) (x_1-x_2) + \left(b_1^2 q_2 - b_2^2 q_1\right)(1+x_1x_2) \Big) \label{eq:Gequiv0:D11=0}
		\end{align}
		From \eqref{eq:Gequiv0:C=0}, we obtain $x_1 = x_2 =:x$, which implies $t_1=t_2 =:t$, since $t_j \in (0,\pi)$. Plugging this into \eqref{eq:Gequiv0:A=0} and \eqref{eq:Gequiv0:B=0}, we have that $0 = b_1^2-b_2^2+x(q_1-q_2)$ and $0 = x(q_1-q_2)-x^2(b_1^2-b_2^2)$. Subtracting these two equations gives $0 = b_1^2-b_2^2+x^2(b_1^2-b_2^2) = (b_1^2-b_2^2)(1+x^2)$. Since $1+x^2>0$, we obtain that $b_1^2 = b_2^2$, and hence $b_1 = b_2 =:b$ since $b_1,b_2>0$.  Plugging these relations into \eqref{eq:Gequiv0:D11=0} and using $\ell>0$, we obtain
		\begin{align*}
		0 = b^2(q_2-q_1)(1+x^2),
		\end{align*}
		implying  $q_1 = q_2$. Thus, we have shown that $(t_1,b_1,q_1) = (t_2,b_2,q_2)$, as desired.
		
		\smallskip
		\begin{subcase}\textbf{\boldmath Exactly one $t_j$ vanishes.} \end{subcase}
		Without loss of generality, assume $t_1=0$ and $t_2\neq 0$; in particular $\sin t_1 = 0$ and $\sin t_2\neq 0$. Since the assumption $t_1 = 0$, $t_2 \neq 0$ implies $B_1 \neq B_2$ and $B_2 \neq I_2$, we aim to show that $G\equiv 0$ is impossible in this case. Using the same substitutions as before, we get
		\[\frac{b_1b_2}{\sin t_2}G_{1,1}
		=
		A\cos^2z+B\sin^2z+Cz\cos z\sin z+D\frac{\cos z\sin z}z,
		\]
		where
		\begin{align*}
		A&=q_1\\
		B&=-(q_2 + (b_1^2 - b_2^2) x_2)\\
		C&=-1/\ell\\
		D&= \ell\left(b_1^2 b_2^2 + q_1 q_2 + (b_1^2 q_2 - b_2^2 q_1) x_2\right)
		\end{align*}
		Since $C = 0$ is clearly impossible, we see that $G_{1,1}$ cannot vanish identically in this case.
		
		\begin{subcase}{\boldmath $ t_1 = t_2 = 0$\textbf{.}}\end{subcase}
		Substituting $w = z/l$, we get
		\[b_1b_2G_{1,1} =  B\sin^2z+D\frac{\cos z\sin z}z,\]
		where $B = b_2^2-b_1^2$ and $D = \ell(b_1^2 q_2 - b_2^2 q_1)$. Thus, if $G_{1,1} \equiv 0$, we obtain $B=D=0$, implying $b_1=b_2$ and $q_1=q_2$ as before.
		
		\begin{case} $\ell_1 \neq \ell_2$ \end{case}

		\begin{subcase}{\boldmath $t_1,t_2\neq 0$\textbf{.}}\end{subcase}
		
		Substituting  $x_j = \cot(t_j)$, we have
		\begin{equation} \label{eq:l1neql2:1}
		\frac{b_1b_2}{w \sin t_1 \sin t_2} G_{1,1}(w)
		=
		A(w)+\frac{1}{w}B(w)+\frac{1}{w^2}C(w)\equiv0
		\end{equation}
		where
		\begin{align*}
		A(w)
		&= x_2 \cos(\ell_1 w) \sin(\ell_2 w)-x_1 \sin(\ell_1 w) \cos(\ell_2 w)\\
		B(w)
		&= \sin(\ell_1 w) \sin(\ell_2 w) (x_2(q_1-b_1^2 x_1)-x_1(q_2 - b_2^2 x_2)) \\ &\qquad +\cos(\ell_1 w) \cos(\ell_2 w) (q_1x_1-q_2x_2+b_1^2-b_2^2)\\
		C(w)&=(b_1^2+q_1 x_1) (q_2-b_2^2 x_2) \cos(\ell_1 w) \sin(\ell_2 w) \\ & \qquad -(q_1-b_1^2 x_1) (b_2^2+q_2 x_2) \sin(\ell_1 w) \cos(\ell_2 w).
		\end{align*}
		Since $B(w)$ and $C(w)$ are bounded functions, \eqref{eq:l1neql2:1} implies $\lim_{w\to\infty}A(w)=0$. Applying Lemma~\ref{lem:triglim} and recalling that $\ell_1\neq\ell_2$, we must have $A(w) \equiv 0$, $x_1=x_2 = 0$, and hence $t_1 = t_2=\frac{\pi}{2}$. Thus, appealing to boundedness of $C(w)$ as before and using \eqref{eq:l1neql2:1} again, we arrive at $B(w) \equiv 0$. Since $x_1=x_2=0$, this implies
		\[
		(b_1^2-b_2^2) \cos(\ell_1 w) \cos(\ell_2 w) \equiv 0,
		\]
		so $b_1=b_2=:b$ (since both are positive). Since $A$ and $B$ vanish identically, appealing to \eqref{eq:l1neql2:1} one more time gives us $C(w)\equiv0$. Consequently,
\[
		q_2 \cos(\ell_1 w) \sin(\ell_2 w)- q_1 \sin(\ell_1 w) \cos(\ell_2 w) \equiv 0.
\]
		From Lemma~\ref{lem:triglim} (using $\ell_1 \neq \ell_2$ again), we have $q_1=q_2=0$.
		
		At this point, we have $t_1=t_2=\frac{\pi}{2}$, $b_1=b_2$, $q_1=q_2=0$. Substituting all of this into the relation $G_{2,1}\equiv 0$, we see that the expression $(b^2-w^2) \sin(w (\ell_1 - \ell_2))$ vanishes for all $w$, which contradicts $\ell_1\neq\ell_2$.
		
		\begin{subcase}\textbf{\boldmath Exactly one $t_j$ vanishes.}\end{subcase}
		Substituting $t_1=0$ and $x_2=\cot t_2$, we get
		\begin{equation}
		\frac{b_1b_2}{w\sin t_2}G_{1,1}
		= D(w)+\frac{1}{w}E(w)+\frac{1}{w^2}F(w),
		\end{equation}
		where
		\begin{align*}
		D(w)&=-\sin(\ell_1 w) \cos(\ell_2 w)\\
		E(w)&=q_1 \cos(\ell_1 w) \cos(\ell_2 w)-\sin(\ell_1 w) \sin(\ell_2 w) (x_2 (b_1^2-b_2^2)+q_2)\\
		F(w)
		&= b_1^2 (b_2^2+q_2 x_2) \sin(\ell_1 w) \cos(\ell_2 w)+q_1 (q_2-b_2^2 x_2) \cos(\ell_1 w) \sin(\ell_2 w).
		\end{align*}
		As before, if $G_{1,1}$ vanishes identically, then, since the functions  $E(w)$ and $F(w)$ are bounded, we get $\lim_{w\to\infty}D(w)=0$, hence $D$ vanishes identically (by Lemma~\ref{lem:triglim}), a contradiction.
		
		\begin{subcase} \textbf{\boldmath $t_1=t_2=0$.}\end{subcase}
		Substituting $t_1=t_2=0$, we get
		\begin{equation}
		b_1 b_2 G_{1,1}=  J(w)+\frac{1}{w}K(w)\equiv0
		\end{equation}
		where
		\begin{align*}
		J(w)
		&=(b_2^2-b_1^2) \sin(\ell_1 w) \sin(\ell_2 w)\\
		K(w)
		&=b_1^2 q_2 \sin(\ell_1 w) \cos(\ell_2 w)-b_2^2 q_1 \cos(\ell_1 w) \sin(\ell_2 w).
		\end{align*}
		As before, $K(w)$ is a bounded function, so $\lim_{w\to\infty}J(w)=0$. Arguing as in previous cases, we have $b_1 = b_2 =:b$ and $J(w)\equiv0$. Consquently, $K(w)\equiv0$, so Lemma~\ref{lem:triglim} yields $q_1=q_2=0$.
		
		Substituting $b=b_1=b_2$, $t_1=t_2=0$, and $q_1=q_2=0$ into $G_{2,1}\equiv 0$, we arrive at
		\[(1-b^2) w \sin(w (\ell_1-\ell_2))\equiv0 \]
		Since $\ell_1\neq \ell_2$,  we must have $b_1=b_2=1$. Therefore $(t_1,b_1,q_1) = (t_2,b_2,q_2) = (0,1,0)$, which implies $B_1=B_2=I_2$, just like we wanted.
	\end{proof}
	
	We are now in a position to prove the Theorem~\ref{t:posExps}.
	
	\begin{proof}[Proof of Theorem~\ref{t:posExps}]
	 	It is enough to check the conditions of \cite[Theorem 2.1]{BuDaFi2} with $A(z):=\cM^{z}(\ell_1, B_1)$ and $B(z):=\cM^{z}(\ell_2, B_2)$. First, both functions are real analytic,  $\tr A(z), \tr B(z)$ are non-constant, and $z\in\bbR$ whenever $\tr A(z), \tr B(z)\in[-2,2]$.  Then we need to show that for some $z_0\in\bbC$ one has
	 	\begin{equation}\lb{2.3a}
	 	[\cM^{z_0}(\ell_1, B_1), \cM^{z_0}(\ell_2, B_2)] \not=0.
	 	\end{equation}
  By Assumption~\eqref{2.3b}, neither \eqref{eq:Gequiv0:1} nor  \eqref{eq:Gequiv0:2} holds and hence Theorem~\ref{t:main} implies the desired result.
	 \end{proof}
	
	\section{Proof of Theorem \ref{thm2.3}}\lb{Section3}

	First, we note that the boundary conditions corresponding to the elements of $\mathscr M_{\sep}$ decouple the operator $H$ and hence lead to localization for somewhat trivial reasons. Specifically, assume that the sequence in \eqref{1.1a} contains a subsequence
	\begin{align}
	\begin{split}\lb{newal}
\{(A_{j_k}, B_{j_k})\}_{k\in\Z}&=\left\{\begin{bmatrix}
x_{j_k}&y_{j_k}\\
0&	0
\end{bmatrix},\begin{bmatrix}
0&0\\
w_{j_k}&	z_{j_k}
\end{bmatrix}\right\}_{k\in\bbZ}\subset{\mathscr M}_{\sep},\\
&j_k\rightarrow\pm \infty,\ k\rightarrow\pm\infty.
	\end{split}
	\end{align}	
	Then
	\begin{equation}
H=\bigoplus_{k\in\bbZ}H_{[j_k, j_{k+1}]},
	\end{equation}
	 where the operator $H_{[m,n]}$ is given by
	\begin{align}
	\begin{split}\lb{2.13}
		&H_{[m,n]}=-d^2/dx^2, H_{[m,n]}:\dom(H_{[m,n]})\subset L^2(t_m, t_n)\rightarrow L^2(t_m, t_n)\\
	&\dom(H_{[m,n]})=\left\{ u\in \hatt H^2(t_{m}, t_{n}): \begin{matrix}
	u\text{\ satisfies \eqref{1.1}\ }\  m< j< n \\
	w_mu(t_{m}^+)+z_mu'(t_{m}^+)=0\\
	x_nu(t_{n}^-)+y_nu'(t_{n}^-)=0
	\end{matrix}\right\}.
	\end{split}
	\end{align}
	Since $H_{[m,n]}$ has compact resolvent, the operator $H$ possesses a basis of compactly supported (hence, exponentially decaying) eigenfunctions and has pure point spectrum.

By Remark \ref{newrem} (ii) it is enough to prove Theorem \ref{thm2.3} (i) assuming \eqref{2.1}, that is, $\supp\wti\mu \cap \mathscr M_{\sep}=\emptyset$ and $\theta=0$ in \eqref{1.1d}. This is accomplished in the following theorem.
	
	\begin{theorem}\lb{main1} Assume Hypothesis \ref{hyp2.1} and recall  $\mathfrak D$ from Theorem \ref{t:posExps}. Then
	there exists a set $\wti \Omega\subset\Omega$ with $\mu(\wti \Omega)=1$ such that for every compact interval $I\subset\bbR\setminus \mathfrak{D}$ and every $\omega\in\wti\Omega$ the following assertions hold:
	\begin{enumerate}
		\item[{\rm(i)}] For every generalized eigenvalue\footnote{A generalized eigenvalue is an energy $E$ admitting a linearly bounded solution, that is, a solution,  $u$, satisfying \eqref{417new}.} $E\in I$ of the operator $H_{\omega}$, one has
		\begin{equation}\lb{a51}
		L(E)=\lim\limits_{n\rightarrow\infty}\frac{1}{n}\log \|M_n^E(\omega)\|=\lim\limits_{n\rightarrow-\infty}\frac{1}{|n|}\log \|M_n^E(\omega)\|.
		\end{equation}
		\item[{\rm(ii)}] The spectral subspace $\ran(\chi_I(H_{\omega}))$ admits a basis of exponentially decaying eigenfunctions.
		\item[{\rm(iii)}] Given $\delta\in(0,1)$ and a normalized eigenfunction
	\begin{equation}
	f\in\ker(H_{\omega}-E), E\in I,\ \|f\|_{L^2(\bbR)}=1,
	\end{equation}
	there exist $\zeta=\zeta(f)\in\bbZ$, $C_{\omega, \delta}>0$, $C_{\delta}>0$ such that
	\begin{equation}\lb{426}
	|f(x^+)|\leq C_{\omega,\delta}e^{C_{\delta}\log^{22}(|\zeta|+1)}e^{-(1-\delta){\overline{L}}(E)|x-\zeta|},\  x\in\bbR.
	\end{equation}
	
	\item[{\rm(iv)}]For every $p>0$ and every compact set $\cK\subset \bbR$ one has
	\begin{equation}
	\sup\limits_{t>0} \left\||X|^p\chi_I(H_{\omega}) e^{-itH_{\omega}}\chi_{\cK}\right\|_{L^2(\bbR)}<\infty.
	\end{equation}
\end{enumerate}
\end{theorem}

\begin{proof} Our argument closely follows the proof of \cite[Theorem 3.11]{DFS} which in turn stems from that of \cite[Theorem 1.2]{BuDaFi}. Throughout the the rest of the proof, $f\lesssim g$ denotes $f\leq C(\mathscr A, I)g$ with some constant $C(\mathscr A, I)>0$ depending only on $\mathscr A$ and $I$.

{\it Proof of  Part {\rm(i)}}.	As was discussed in \cite[Section 3.3]{DFS}, Theorem 2.1 yields a Large Deviation Theorem, \cite[Theorem 3.1]{BuDaFi}, which in turn implies the following two facts:\newline
 $\bullet$ For every $0<\varepsilon<1$, there exists a full-measure set $\Omega_1(\varepsilon) \subseteq \Omega$ with the following property: For every $\omega \in \Omega_1(\varepsilon)$, there is $n_1 = n_1(\omega,\varepsilon)$ such that
\begin{equation}\lb{413new}
\left| L(E) - \frac{1}{n^2} \sum_{s=0}^{n^2-1}\frac{ \log\|M^E_n(T^{\zeta+sn}\omega)\|}{n} \right|
<
\varepsilon
\end{equation}
for all  $E \in I$, $\zeta\in\bbZ$, and $n \geq \max ( \log^{\frac23}(|\zeta|+1),n_1)$, see \cite[Proposition 5.2]{BuDaFi}.
\newline
$\bullet$ For every $0<\varepsilon<1$, there exists a full measure set $\Omega_{2}(\varepsilon)$ such that for every $\omega\in\Omega_2(\varepsilon)$, there is $n_2=n_2(\varepsilon,\omega)$ such that
\begin{equation}\lb{45new}
\frac{1}{n}\log \|M_n^{E}(T^{\zeta_0}\omega)\|\leq L(E)+\varepsilon
\end{equation}
for any $\zeta_0\in\bbZ$ and $n\geq \max( \log ^2(|\zeta_0|+1),n_2)$, see \cite[Corollary~5.3]{BuDaFi}. In particular, this yields
\begin{equation}\lb{46new}
\mu\set{\omega: \text{for all } E\in I,\  \limsup\limits_{n\rightarrow\infty}\frac{1}{n}\log \|M_n^{E}(\omega)\|\leq L(E)}=1. \
\end{equation}

Our  objective is to show that
\begin{equation}\lb{47new}
\mu\set{\omega \; : \; \begin{aligned}
&  \liminf\limits_{n\rightarrow\infty}\frac{1}{n}\log \|M_n^{E}(\omega)\|\geq L(E) \\
& \quad \text{for all {\it generalized eigenvalues}}\  E\in I \end{aligned} }=1.
\end{equation}
To that end, we first note a version of \cite[Theorem 3.10]{DFS}\footnote{The proof of this fact for the model in question is almost identical to that of \cite[Theorem 3.10]{DFS}.} concerning the elimination of double resonances. Denote the Neumann restriction of $H_{\omega}$ to $[t_m, t_n]$ by $H_{[m,n]}(\omega)$\footnote{cf. \eqref{2.13} with $w_m=x_n=0$, $z_m=y_n=1$}. For $\varepsilon\in(0,1)$ and $N\in\bbN$, define
\begin{align}
&\cD_N(\varepsilon)
:=
\set{
\omega\ \Bigg| \ \begin{aligned}
&\qquad\quad\text{there exist\ } E\in I,\ \zeta\in\bbZ_+,\\
&K\geq \max\{\log^2(\zeta+1) , N\},\  0\leq N_1, N_2\leq K^9\text{\ such that:\ }\\
&\begin{cases}
\  |F_m(T^{r+\zeta}\omega,E)|\leq L(E)-\varepsilon\\
\text{and\ } \|(H_{[-N_1, N_2]} (T^\zeta\omega)-E)^{-1}\|\geq e^{K^2}\\
\text{for some\  }  m\in\{K,2K\},\ K^{10}\leq r\leq \overline{K},\\
\text{where } \overline{K} := \left\lfloor K^{\log K} \right\rfloor
\end{cases}
\end{aligned}
}
\end{align}
and $\Omega_3(\varepsilon):=\Omega\setminus\limsup\limits_{N\rightarrow\infty}\cD_N(\varepsilon)$; then one has $\mu(\Omega_3(\varepsilon))=1.$
Define the full measure set
\begin{equation}
\wti\Omega:= \bigcap_{0 < \varepsilon < \varkappa}\ \bigcap_{j=1}^3 \Omega_j(\varepsilon),
\end{equation}
where
\[
\varkappa:= \frac{1}{3} \min_{E \in I} L(E).
\]

Fix $\omega=\{\ell_j,  (I_2, B_j)\}_{j\in\bbZ}\in\wti\Omega$ and a generalized eigenvalue $E\in I$. Then in order  to establish \eqref{47new} it is enough to check
\begin{equation}\lb{416new}
\liminf\limits_{n\rightarrow\infty}\frac{1}{n}\log \|M_n^{E}(\omega)\|\geq L(E).
\end{equation}
Now, let $u$ denote a generalized eigenfunction corresponding to the generalized eigenvalue $E$ and the operator $H_\omega$. That is,
\begin{align}
\begin{cases}-u''=Eu,\\
\text{\ for some\ } C_u>0,\ \max\set{|u'(t_j^{\pm})|,|u(t_j^{\pm})|}\lesssim C_u(1+|j|), \text{\ for all } j\in\bbZ,  \\
\text{and\ } \text{$u$ satisfies the vertex conditions~\eqref{1.1} for all\ }j\in\bbZ.
\end{cases}\lb{417new}
\end{align}
We now follow the blueprint of \cite{DFS}. Given $0<\varepsilon<\varkappa$, the primary goal is to show that
\begin{equation}\lb{418}
\frac{1}{n}\log \|M_n^{E}(\omega)\|\geq L(E)-6\varepsilon,\text{\ whenever $K^{11}+K^{10} \leq n \leq \overline{K}$ }
\end{equation}
for all sufficiently large $K$. The intervals so described cover a half-line, and hence \eqref{418} implies \eqref{416new}.

Given $\zeta\in\bbZ$, define\footnote{In the arguments that follow, $\zeta$ will correspond to the center of localization.}
\begin{equation}\lb{349nn}
K(N):=
\max
\set{\lceil\log^2(|\zeta|+1)\rceil,  n_1, n_2, n_3,N},
\end{equation}
where $N\in\bbN$ will be determined later\footnote{Specifically, $N$ will depend solely on $C_u$, so, if all generalized eigenfunctions are bounded, then $N$ may be chosen independently of $u$.}, $n_1$ and $n_2$ are as discussed near \eqref{413new} and \eqref{45new} (respectively), and $n_3=n_3(\omega, \varepsilon)$ is the minimal integer such that
\begin{equation}\lb{431}
\omega\in \Omega \setminus \cD_j(\varepsilon) \text{ for every } j \geq n_3.
\end{equation}
\begin{step}
There exists $N=N(C_u)>0$ such that for all $K\geq K(N)$, there exist integers $m_1\in[\zeta-K^9, \zeta], m_2\in[\zeta, \zeta+K^9]$ such that
\begin{equation}\lb{350nn}
|u({t_{m_j}^-})|\leq e^{-2K^2},\ |u'({t_{m_j}^-})|\leq e^{-2K^2}
\end{equation}
for $j=1,2.$
\end{step}

\begin{proof}
Using  \eqref{413new} with $n=K^3$ and $\zeta:=\zeta-K^9$ we get
\begin{equation}\lb{421}
L(E) - \frac{  \log\|M^E_{K^3}(T^{\zeta+sK^3}\omega)\|}{K^3} < \varepsilon
\end{equation}
for some $s=:s_1\in[-K^9, -K^3]\cap \bbZ$. Thus,
\begin{equation}\lb{425}
\exp((L(E) -\varepsilon)K^3)<  \|M^E_{K^3}(T^{\zeta+sK^3}\omega)\|
\end{equation}
Likewise, using \eqref{413new} with $n=K^3$, we obtain \eqref{425} for some $s_2\in[0, K^6-1]\cap \bbZ$. Fixing such an $s_2$, we introduce $\alpha$ and $\beta$ via
\begin{equation}\lb{425a}
[\alpha, \beta]:=[\zeta+s_2K^3, \zeta+(s_2+1)K^3],\ m_2:= \left\lfloor\frac{\alpha+ \beta}{2}\right\rfloor.
\end{equation}
We will show that this choice of $m_2$ gives \eqref{350nn} with $j=2$. The proof for $j=1$ relies on \eqref{425} with $s=s_1$ and is completely analogous. Our argument is based on a representation of $u$ in terms of its boundary values $u(t_\alpha^+)$, $u(t_\beta^-)$ and special solutions satisfying specific boundary conditions chosen based on which entry of the matrix
\begin{equation}\lb{356nn}
B_{\beta}^{-1}M^E_{K^3}(T^{\alpha}\omega)
\end{equation}
dominates its norm. 

Thus, there are four cases; we will consider one case and note that the other cases are completely similar. The reader may also consult \cite{DFS} to see what modifications one should make in the other three cases.

To that end, let $m_{ij}$ denote the $ij$ entry of \eqref{356nn}, and suppose $\|B_{\beta}^{-1}M^E_{K^3}(T^{\alpha}\omega)\|\leq 4 |m_{11}|$. In this case, we choose $\psi_\pm$ to satisfy the interior vertex conditions as well as the boundary conditions
\begin{equation}
 \psi'_+({{t_{\beta}^-}})=1,\ \psi_+({{t_{\beta}^-}})=0,\
\psi'_-({{t_{\alpha}^+}})=0, \ \psi_-({{t_{\alpha}^+}})=1,\
\end{equation}
and observe that
\begin{equation}\lb{437n}
|W(\psi_+,\psi_-)|=|\psi_+'({{t_{\alpha}^+}})|=|\psi_-({{t_{\beta}^-}})|=|m_{11}|>0.
\end{equation}
By \eqref{437n}, $\psi_-$ and $\psi_+$ are linearly independent, so we may write
\begin{equation}\lb{436n}
u({t_{m_2}^-})=u'({{t_{\alpha}^+}})\frac{\psi_+({t_{m_2}^-})}{\psi_+'({{t_{\alpha}^+}})}+u({{t_{\beta}^-}})\frac{\psi_-({t_{m_2}^-})}{\psi_-({{t_{\beta}^-}})}.
\end{equation}



Next, we estimate the right-hand side of \eqref{436n}. Putting together \eqref{425} and \eqref{437n}, we obtain
\begin{align}
\begin{split}\lb{447n}
|\psi_+'({{t_{\alpha}^+}})|=|\psi_-({{t_{\beta}^-}})|=|m_{11}|&\geq\frac{\|B_{\beta}^{-1}M^E_{K^3}(T^{\alpha}\omega)\|}{4}\\
&\geq \frac{ \|M^E_{K^3}(T^{\alpha}\omega)\|}{4\|B_{\beta}^{-1}\|}\\
&\gtrsim \exp((L(E) -\varepsilon)K^3).
\end{split}
\end{align}
Next, \eqref{417new} implies
\begin{equation}\lb{425new}
\max\set{|u'({{t_{\alpha}^+}})|, |u({{t_{\beta}^-}})}|\lesssim C_u(K^9+e^{\sqrt {K}}).
\end{equation}
Next, apply \eqref{45new} with $\zeta_0=\zeta+s_2 K^3$ and $n=\lfloor \frac{K^3}{2}\rfloor$, and select $N$ so that $\lfloor \frac{K^3}{2}\rfloor\geq\log^2(\zeta+sK^3)$ to get
\begin{align}
\begin{split}
|{\psi_-({t_{m_2}^-})}|&\leq \left|\left\langle \begin{bmatrix}1 \\ 0 \end{bmatrix}, B^{-1}_{m_2}\,M_{\lfloor \frac{K^3}{2}\rfloor}^E(T^{\zeta+sK^3}\omega)\begin{bmatrix} 1 \\ 0\end{bmatrix} \right \rangle\right|\\
&\lesssim \exp\left(\frac{1}{2} (L(E)+\varepsilon) K^3 \right).\lb{449new}
\end{split}
\end{align}
Similarly, for $N$ sufficiently large, we get
\begin{align}\lb{427}
\begin{split}
\left|\psi_+({t_{m}^-})\right|\lesssim \exp\left( \frac{1}{2} (L(E)+\varepsilon)K^3\right).
\end{split}
\end{align}
Putting together \eqref{436n}, \eqref{447n}, \eqref{449new}, and \eqref{427}, we have
\begin{align}\lb{428}
|u({t_{m}^-})|
\lesssim 2C_u(K^9+e^{\sqrt {K}})\exp\left(-\frac{1}{2}(L(E) - 3\varepsilon)K^3\right)\leq e^{-2K^2},
\end{align}
where the final inequality holds for $N=N(C_u, \mathscr A)$ sufficiently large. Similarly,
\begin{equation}\lb{452n}
|u'({t_{m_2}^-})|\leq e^{-2K^2},
\end{equation}
follows by replacing $u({t_{m_2}^-})$ (respectively, $\psi_{\pm}({t_{m_2}^-})$) by $u'({t_{m_2}^-})$ (respectively, $\psi'_{\pm}({t_{m_2}^-})$) in \eqref{436n}, and $[1,0]^{\top}$ by $[0,1]^{\top}$ in both \eqref{449new} and \eqref{427}.
\end{proof}
\begin{step}
If $|u(\tau^-)|=1$ for some $\tau\in\bbR$, let $\zeta$ be the largest integer for which $t_{\zeta}<\tau$. If $|u(\tau^+)|=1$ for some $\tau\in\bbR$, let $\zeta$ be the largest integer for which $t_{\zeta}\leq \tau$. Let $m_1, m_2$ be as in Step~1. Then
\begin{equation}\lb{3.73newnew}
\|(H_{[m_1, m_2]}({\omega})-E)^{-1}\|_{\cB(L^2(t_{m_1}, t_{m_2}))}\geq e^{K^2}.
\end{equation}
\end{step}
\begin{proof}
There exists a $K$-independent interval $J\subset (t_{\zeta}, t_{\zeta+1})$ such that
\begin{equation}\lb{321a}
1/2\leq  |u(x)|\ \text{for all }  x\in J.
\end{equation}
Suppose that $\psi_{\pm}$ satisfies $-\psi_{\pm}''=E\psi_{\pm}$ in $(t_{m_1},t_{m_2})$, the interior vertex conditions in the interval $(t_{m_1},t_{m_2})$ and
\begin{align}
&\psi_-({{t_{m_1}^+}})=1,\ \psi'_-({{t_{m_1}^+}})=0,\ \psi_+({{t_{m_2}^-}})=1,\ \psi'_+({{t_{m_2}^-}})=0.
\end{align}
Then one has
\begin{align}
\begin{split}
u(x)&=u'(t_{m_1}^+)\frac{\psi_+(x)}{W(\psi_+,\psi_-)}+u'(t_{m_2}^-)\frac{\psi_-(x)}{W(\psi_+,\psi_-)},\  x\in(t_{\zeta},t_{\zeta+1}).\\
\end{split}
\end{align}
Integrating over $J$, using \eqref{321a} and \eqref{350nn} we infer
\begin{align}
\begin{split}
\frac{|J|}{2}&\leq |u'(t_{m_1}^+)|\frac{\int_J|\psi_+(x)|dx}{|W(\psi_+,\psi_-)|}+|u'(t_{m_2}^-)|\frac{\int_J|\psi_-(x)|dx}{|W(\psi_+,\psi_-)|},\\
&= e^{-2K^2}\frac{\int_J|\psi_+(x)|dx}{|W(\psi_+,\psi_-)|}+e^{-2K^2}\frac{\int_J|\psi_-(x)|dx}{|W(\psi_+,\psi_-)|}.\lb{4.53new}
\end{split}
\end{align}
Thus, without loss we may assume
\begin{align}
\begin{split}\lb{4.53newa}
\frac{e^{2K^2}|J|}{4}\leq \frac{\int_J|\psi_-(x)|dx}{|W(\psi_+,\psi_-)|}.
\end{split}
\end{align}
Since $\psi_+(y)=\cos(\sqrt{E}(y-t_{m_2}))$, $y\in (t_{m_2-1}, t_{m_2}]$ we have
\begin{align}\no
&\psi_+(y)\geq 1/2,\ y\in(t_{m_2}-\rho, t_{m_2}],
\end{align}
for a suitable $K$-independent constant $\rho>0$ which is sufficiently small.
Combining the previous two inequalities we get
\begin{align}
\begin{split}
\frac{e^{2K^2}|J|\rho}{8} & \leq \int_{t_{m_2}-\rho}^{ t_{m_2}} \int_J\frac{|\psi_-(x)|\psi_+(y)}{|W(\psi_+,\psi_-)|} \, dx \, dy.
\end{split}
\end{align}
Let   $G_{\omega, [m_1, m_2]}(x,y)$ denote the Green function of $H_{[m_1, m_2]}{(\omega)}$, then
\begin{equation}
G_{\omega, [m_1, m_2]}(x,y)=\frac{\psi_-(x)\psi_+(y)}{W(\psi_+,\psi_-)},\ x\in J, y\in (t_{m_2}-\rho, t_{m_2}).
\end{equation}
Denoting
\begin{equation}
\phi_1:=\chi_{J}\sign(\psi_+ W(\psi_+,\psi_-)),\ \phi_2:=\chi_{(t_{m_2}-\rho, t_{m_2})},
\end{equation}
we get
\begin{align}
e^{K^2}&\leq \frac{|\langle\phi_1, (H_{[m_1, m_2]}{(\omega)}-E)^{-1}\phi_2\rangle_{L^2(t_{m_1}, t_{m_2})}|}{\|\phi_1\|_{L^2(t_{m_1}, t_{m_2})}\|\phi_2\|_{L^2(t_0, t_{m})}}\\
&\leq \|(H_{[m_1, m_2]}{(\omega)}-E)^{-1}\|_{\cB(L^2(t_{m_1}, t_{m_2}))},
\end{align}
for  sufficiently large $N$ in \eqref{349nn}.
\end{proof}

\begin{step} Choose $\zeta$ as in the previous step. For a suitable $N=N(C_u)$, we have
\begin{equation}\lb{443}
\frac{1}{n}\log \|M_n^{E}(T^{\zeta}\omega)\|\geq L(E)-6\varepsilon.
\end{equation}
for every $K\geq K(N)$ and every $K^{11}+K^{10} \leq n\leq  \overline{K}$.
\end{step}
\begin{proof}
By \eqref{431} and \eqref{3.73newnew}, we have
\begin{equation}
\frac{1}{jK}\log \|M_{jK}^{E}(T^{\zeta+r}\omega)\|\geq L(E)-\varepsilon,\ \text{ for each } j=1,2 \text{ and all } K^{10} \leq r\leq \overline{K}.
\end{equation}
This input suffices to apply the Avalanche Principle \cite{GS01} to deduce \eqref{443}. Consult \cite[(6.17)--(6.18)]{BuDaFi} for details.
\end{proof}

Since the intervals in Step~3 cover a half-line, \eqref{416new} and \eqref{a51} follow immediately.

{\it Proof of  Part {\rm(ii)}.}  By Part~(i) and Osceledets'~Theorem, every generalized eigenvalue of $H_\omega$ is an eigenvalue whose corresponding eigenfunctions decay exponentially at the rate dictated by the Lyapunov exponent.

Next, we note that
\begin{equation}
(H_\omega-\bfi)^{-1}\in\cB(L^2(\bbR), L^{\infty}(\bbR)).
\end{equation}
Indeed, assume that $u\in\dom(H_{\omega})$, $-u''-\bfi u=f\in L^2(\bbR)$. Then by Sobolev inequalities  (see, e.g,  \cite[Corollary~4.2.10]{Bu}, \cite[IV.1.2]{K80}) we get
\begin{align}
\begin{split}
&\|u\|_{L^{\infty}(t_j,t_{j+1})}\lesssim {\|u\|_{L^2(t_j,t_{j+1})}}+\|u''\|_{L^2(t_j,t_{j+1})}=\|f\|_{L^2(t_j,t_{j+1})}, j\in\bbZ.
\end{split}
\end{align}
Therefore \cite[Assumption B.1]{HiPo} is satisfied and by \cite[Theorem B.9]{HiPo} the spectral measure of $H_{\omega}$ is supported by the generalized eigenvalues. Thus the eigenfunctions corresponding to generalized eigenvalues in $I$ span the spectral subspace $\ran(\chi_I(H_{\omega}))$.

{\it Proof of  Part {\rm(iii)}.} First, using $-f''=Ef$ and the one-dimensional Sobolev inequalities (see, e.g,  \cite[Corollary~4.2.10]{Bu}, \cite[IV.1.2]{K80}), we have
\begin{align}
\begin{split}\lb{464n}
&\max\set{\|f\|_{L^{\infty}(t_j,t_{j+1})}, \|f'\|_{L^{\infty}(t_j,t_{j+1})}} \\
&\quad\lesssim {\|f\|_{L^2(t_j,t_{j+1})}}+\|f''\|_{L^2(t_j,t_{j+1})} \\
& \quad \lesssim 1,
\end{split}
\end{align}
and
\begin{align}
\begin{split}\lb{377new}
\|f'\|_{L^{\infty}(t_{j}, t_{j+1})}&\lesssim\|f \|_{L^{2}(t_{j}, t_{j+1})}+\|f''\|_{L^{2}(t_{j}, t_{j+1})}\\
&\lesssim\|f \|_{L^{2}(t_{j}, t_{j+1})} \\
&\lesssim \|f \|_{L^{\infty}(t_{j}, t_{j+1})}.
\end{split}
\end{align}
Therefore, after taking
\begin{equation}
\begin{split}\lb{467n}
&u=\frac{f}{\|f\|_{L^{\infty}(\bbR)}},\ C_u\lesssim 1\,\text{\ in Step 1},\\
& \text{and $\tau$  in Step~2 such that either $|f(\tau^-)| = \|f\|_\infty$ or $|f(\tau^+)| = \|f\|_\infty$},
\end{split}
\end{equation}
 we may repeat the arguments from Part~(i). We pick any value of $\tau$ with desired property and note that its existence is guaranteed because
\begin{align}
\set{\begin{bmatrix}
	f(t_j^+)\\
	f'(t_j^+)
	\end{bmatrix}}_{j\in\bbZ}\in \ell^2(\bbZ,\C^2)\text{\ and\ }\lim\limits_{|t|\to\infty}|f(t)| = \lim\limits_{|t|\to\infty}|f'(t)|=0.
\end{align}
So, given $0 < \varepsilon < \varkappa$, we may choose $N=N(\varepsilon, \omega)$ (independent of $f$) so that
\begin{equation}
\frac{1}{n}\log \|M_n^{E}(T^{\zeta}\omega)\|\geq L(E)-6\varepsilon
\end{equation}
 for all $K\geq \max\{ K(N), \log^2(|\zeta|+1)\}$ and all $K^{11}+K^{10} \leq n \leq \overline{K}$.
Our next objective is to show that
\begin{equation}\lb{459n}
|f(t_{\zeta+n}^+)|\lesssim  e^{-(1-\delta)L(E)n}, \text{\ for all\  }n\in\left[\frac{p}{4}, \frac{p-1}{2}\right],
\end{equation}
for all $K\geq K(N)$ and  $K^{11}+K^{10} \leq p \leq \overline{K}$. As before, the proof relies on a representation of $f$ in terms of its boundary values at $t_\zeta$ and $t_{\zeta+p}$, and the choice of representation, depends on which entry of
\begin{equation}
B_{\zeta+p}^{-1}M_p^{E}(T^{\zeta}\omega)
\end{equation}
dominates its norm. We will argue under the assumption that the upper left entry dominates the norm; the other three cases are almost identical.

Choose $\psi_\pm$ as follows: they satisfy the interior vertex conditions in $[t_\zeta,t_{\zeta+p}]$, solve $-\psi_{\pm}''=E\psi_{\pm}$, and satisfy the boundary conditions
\begin{align}
&\psi_-(t_\zeta^+)=1,\ \psi'_-(t_\zeta^+)=0, \psi_+(t_{\zeta+p}^-)=0,\ \psi'_+(t_{\zeta+p}^-)=1.
\end{align}
Notice that
\begin{align}
\begin{split}
|W(\psi_+,\psi_-)|&=|\psi_+'(t_\zeta^+)|=|\psi_-(t_{\zeta+p}^-)|\\
&\geq\frac{\|B_{\zeta+p}^{-1}M_p^{E}(T^{\zeta}\omega)\|}{4}\\
&\geq\frac{\|M_p^{E}(T^{\zeta}\omega)\|}{4\|B_{\zeta+p}\|}\\
& \gtrsim \exp((L(E) -6\varepsilon)p).\lb{463n}
\end{split}
\end{align}

One has
\begin{align}\lb{459}
\frac{f(t_{\zeta+n}^+)}{M_f}=\frac{f'(t_\zeta^{+})\psi_+(t_{\zeta+n}^+)}{M_f\psi_+'(t_\zeta^+)}+\frac{f(t_{\zeta+p}^-)\psi_-(t_{\zeta+n}^+)}{M_f\psi_-(t_{\zeta+p}^-)},
\end{align}
where  $M_f:=\|f \|_{L^{\infty}(\bbR_+)}$.
In order to estimate $\psi_{-}(t_{\zeta+n}^+)$, we rewrite it in terms of the transfer matrices and use \eqref{45new} as follows
\begin{equation}
|\psi_{-}(t_{\zeta+n}^+)|=\left|\left\langle \begin{bmatrix}1 \\ 0 \end{bmatrix} , M_{n}^E(T^{\zeta}\omega)\begin{bmatrix} 1 \\ 0\end{bmatrix} \right \rangle\right|\leq \exp((L(E) +\varepsilon)n).
\end{equation}
Similarly one can estimate $\psi_{+}(t_{\zeta+n}^+)$.
Put this together with \eqref{464n}--\eqref{459} to obtain
\begin{align}
\begin{split}
|f(t_{\zeta+n}^+)|& \lesssim \exp((L(E) +\varepsilon)n-(L(E) -6\varepsilon)p)\\
&\quad\quad+\exp((L(E) +\varepsilon)(p-n)-(L(E) -6\varepsilon)p)\\
& \lesssim  2 e^{-(1-\delta)nL(E)}.
\end{split}
\end{align}
In the final inequality, pick $\varepsilon=\varepsilon(\delta)>0$ sufficiently small (which only depends on $\delta$).  Thus
\begin{equation}\lb{469n}
|f(t_{\zeta+n}^+)|\lesssim  e^{-(1-\delta)L(E)n},
\end{equation}
for all $K\geq K(N)$ and $\frac{K^{11}+K^{10}}{4} \leq n \leq \frac{\overline{K}-1}{2}$. So, for sufficiently large $N$, the inequality in \eqref{469n} holds for all
\begin{equation}
n\geq
\frac{1}{2} \max\set{\log^2(|\zeta|+1), N(\omega,\varepsilon)}^{11} =:R.
\end{equation}
Trivially estimating $f(t_{\zeta+n}^+)$ for $0 \le n \le R$ (that is, using $\|f\|_{L^{\infty}(\bbR)}\lesssim 1$) we obtain
\begin{align}
\begin{split}\lb{3102n}
|f(t_{\zeta+n}^+)| &\lesssim C_{\omega,\delta}e^{C_{\delta}\log^{22}(|\zeta|+1)}e^{-(1-\delta)L(E)n},\ n\geq 0.
\end{split}
\end{align}

Using the representation
\begin{align}
\frac{f'(t_{\zeta+n}^+)}{M_f}=\frac{f'(t_\zeta^{+})\psi'_+(t_{\zeta+n}^+)}{M_f\psi_+'(t_\zeta^+)}+\frac{f(t_{\zeta+p}^-)\psi'_-(t_{\zeta+n}^+)}{M_f\psi_-(t_{\zeta+p}^-)},
\end{align}
we may prove a version of \eqref{3102n} with $f$ replaced by $f'$ by repeating \eqref{463n}--\eqref{3102n}.
Lastly, we infer \eqref{426} for all $x\geq 0$ by interpolation. The same argument applies to the negative half-axis.

{\it Proof of Part {\rm (iv)}.} The argument is essentially identical to the proof of \cite[Theorem 1.1]{DFS}, which in turn stems from  \cite{GdB},  with natural substitution of one-sided intervals by their symmetric two-sided versions.
\end{proof}

	

\begin{thebibliography}{99}
		\bi{ADK} S.\ Albeverio, L.\ Dabrowski, P.\ Kurasov,  {\it Symmetries of Schr\"odinger operators with point interactions}, Lett. Math. Phys. {\bf 45} (1998), 33--47.
		%
		\bi{AGHH} S.\ Albeverio, F.\ Gesztesy, R.\ Hoegh-Krohn,\  H.\ Holden,  with app. by P.\ Exner, {\em  Solvable Models in Quantum Mechanics},  2nd edition, AMS-Chelsea Series, Amer. Math. Soc., 2005.
		%
		
		\bi{AFK} S.\ Albeverio, S.\ Fei, P.\ Kurasov, {\em Many body problems with ``spin"-related contact interactions}, Rep. Math. Phys. {\bf 47} (2001), 157--166.		
		%
		\bi{BF} F.\ A.\ Berezin, L.\ D.\ Faddeev,  {\it Remark on the Schr\"odinger equation with singular potential} (in Russian), Dokl. Akad. Nauk SSSR {\textbf 137} (1961), 1011--1014.
		%
		\bi{BuDaFi} V.\ Bucaj,\ D.\ Damanik,\ J.\ Fillman,\ V.\ Gerbuz,\ T.\ VandenBoom,\ F.\ Wang,\ Z.\ Zhang, {\it Localization for the one-dimensional Anderson model via positivity and large deviations for the Lyapunov exponent}, Trans. Amer. Math. Soc. {\bf 372} (2019), 3619--3667.
		%
		\bi{BuDaFi2} V.\ Bucaj,\ D.\ Damanik,\ J.\ Fillman,\ V.\ Gerbuz,\ T.\ VandenBoom,\ F.\ Wang,\ Z.\ Zhang, {\it Positive Lyapunov exponents and a large deviation theorem for continuum Anderson models, briefly}, J. Funct. Anal., in press, arXiv:1902.04642. DOI: https://doi.org/10.1016/j.jfa.2019.05.028
		%
		\bi{Bu} V.\ I.\  Burenkov, {\em Sobolev Spaces on Domains}, B.G. Teubner, Stuttgart--Leipzig, 1998.
		%
		\bi{CH} P.\ Chernoff, R.\ Hughes, {\it A new class of point interactions in one dimension}, J. Funct. Anal. {\bf 111} (1993), 97--117.		
		%
		\bibitem{DLS} D.\ Damanik, L.\ Fang, S.\ Sukhtaiev, {\it Zero measure and singular continuous spectra for quantum graphs}, arXiv:1906.02088.
		%
		\bibitem{DFS} D.\ Damanik, J.\ Fillman, S.\ Sukhtaiev, {\it Localization for Anderson models on metric and discrete tree graphs}, arXiv:1902.07290.
		%
		\bi{DSS} F.\ Delyon, B.\ Simon, B.\ Souillard, {\it From power pure point to continuous spectrum in disordered systems},  Ann. Inst. H. Poincare Phys. Theor. \textbf{42} (1985), 283--309.
		%
		\bi{FS83} J.\ Fr\"ohlich, T.\ Spencer, {\it Absence of diffusion in the Anderson tight binding model for large disorder or low energy}, Commun. Math. Phys. {\bf 8} (1983), 151--184.
		%
		\bi{F63} H.\ F\"urstenberg, {\it Noncommuting random products}, Trans. Amer. Math. Soc. {\bf 108} (1963), 377--428.
		%
		\bi{GdB} F.\ Germinet, S.\ De Bi\'evre, {\it Dynamical localization for discrete and continuous	random Schr\"odinger operators}, Commun. Math. Phys. {\bf 194} (1998), 323--341.
		%
		\bibitem{GS01} M.\ Goldstein, W.\ Schlag, {\it H\"older continuity of the integrated density of states for quasi-periodic Schr\"odinger equations and averages of shifts of subharmonic functions}, Ann. of Math. (2) {\bf 154} (2001), 155--203.
		%
		\bi{HKK05} P.\ Hislop, W.\ Kirsch, M.\ Krishna, {\it Spectral and dynamical properties of random models with nonlocal and singular interactions}, Math. Nachr. {\bf 278} (2005), 627--664.
		%
		\bi{HKK19} P.\ Hislop, W.\ Kirsch, M.\ Krishna, {\it Eigenvalue statistics for Schr\"odinger operators with random point interactions on $\bbR^d$, $d=1,2,3$}, arXiv:1905.0788.
		%
		\bi{HiPo} P.\ Hislop, O.\ Post, {\it Anderson localization for radial tree-like quantum graphs}, Waves Random Complex Media {\bf 19} (2009), 216--261.
		%
		\bi{K80} T.\ Kato, {\it Perturbation Theory for Linear Operators}, Springer, Berlin, 1980.
		%
		\bi{Ki08} W.\ Kirsch, {\it An invitation to random Schr\"odinger operators}, Panor. Synth\`eses, 25, Random Schr\"odinger operators, 1--119, Soc. Math. France, Paris, 2008.
		%
		\bi{LangSL2R} S.\ Lang, $\SL_2(\R)$, Springer, New York, 1985.
		%
		\bi{K96} P.\ Kurasov, {\it Distribution theory for discontinuous test functions and differential operators with generalized coefficients}, J. Math. Anal. Appl. {\bf 201} (1996), 297--323.		
		%
		\bi{Seba} P.\ Seba, {\it The generalized point interaction in one dimension}, Czechoslovak J. Phys. B {\bf 36} (1986), 667--673.
		%
		\bi{St01} P.\ Stollmann  {\em Caught by Disorder, Bound States in Random Media}, Progress in Mathematical Physics 20, Birkh\"auser, Boston, 2001.		
	\end{thebibliography}
	\end{document}